\documentclass[reqno,11pt]{amsart}
\usepackage{lineno,amsmath,amsfonts,amsthm,amssymb,graphics,epsfig,mathrsfs, eufrak}
\usepackage{enumerate,geometry,etoolbox, scalerel, stackengine, relsize}
\usepackage{amscd,latexsym,url,upgreek,mathtools,color}
\usepackage{mathrsfs}
\usepackage{hyperref}
\usepackage[english]{babel}

\usepackage[utf8]{inputenc}
\usepackage[T1]{fontenc}

\usepackage{tikz}
\usepackage{booktabs}

\stackMath
\newcommand\reallywidehat[1]{%
\savestack{\tmpbox}{\stretchto{%
  \scaleto{%
    \scalerel*[\widthof{\ensuremath{#1}}]{\kern-.6pt\bigwedge\kern-.6pt}%
    {\rule[-\textheight/2]{1ex}{\textheight}}
  }{\textheight}%
}{0.5ex}}%
\stackon[1pt]{#1}{\tmpbox}%
}

\newtheorem{theorem}{Theorem}[section]
\newtheorem{lemma}[theorem]{Lemma}

\theoremstyle{definition}
\newtheorem{remark}[theorem]{{\bf Remark}}
\newtheorem{definition}[theorem]{Definition}

\title{On the approximation of Weierstrass function via superoscillations}

\author[F. Colombo]{F. Colombo}
\address{(FC) Politecnico di
Milano\\Dipartimento di Matematica\\Via E. Bonardi, 9\\20133 Milano\\Italy}
\email{fabrizio.colombo@polimi.it}

\author[I. Sabadini]{I. Sabadini}
\address{(IS) Politecnico di
Milano\\Dipartimento di Matematica\\Via E. Bonardi, 9\\20133 Milano\\Italy}
\email{irene.sabadini@polimi.it}

\author[D. C. Struppa]{D. C. Struppa}
\address{(DCS) The Donald Bren Presidential Chair in Mathematics\\ Chapman University, Orange, CA 92866 \\ USA}
\email{struppa@chapman.edu}

\thanks{F. Colombo and  I. Sabadini are supported by MUR grant Dipartimento di Eccellenza 2023-2027. D.C. Struppa thanks the Parker Kennedy Chair in Mathematics for research support.}

\begin{document}

\maketitle

\begin{abstract}
The Weierstrass function is a classic example of a continuous nowhere differentiable function, defined as a sum of high-frequency complex exponentials. In this paper, we follow a suggestion of M.V. Berry and study the convergence properties of Berry's superoscillating approximation to the truncated Weierstrass function. We provide sharp, explicit error estimates for this approximation and we analyze the subtle convergence properties of the associated double limits.
\end{abstract}

\medskip
AMS Classification 35A20

\medskip
Keywords: Superoscillations,
Weierstrass function,
Truncated Weierstrass function.

\section{Introduction}

Superoscillations have become, over the last 10--20 years, a significant object of study for both mathematicians \cite{colombo2017,BOREL,Pozzi} and physicists interested in quantum mechanics and optics \cite{AAV, aharonov1990,berry1994,Ahar-bohm-JMP,kempf2,kempf3,jordan,jordan2,kempf4}. A comprehensive introduction to the state of the art, at least as of a few years ago, is the Roadmap on Superoscillations \cite{berry2019roadmap}.

Roughly speaking, superoscillations are functions that, within a small interval, oscillate much faster than their highest Fourier frequency component. This apparently paradoxical property results from almost perfect destructive interference among the frequencies, and is made possible by the exponentially fast growth of such functions as soon as one leaves the superoscillation region.

The prototypical example of superoscillations (found in virtually every article on the subject) is given by the following expression, where $n$ is an integer and $\alpha$ is a large positive real parameter:
\begin{equation}\label{form1}
F_n(x;\alpha) := \left(\cos \frac{x}{n} + i\alpha \sin \frac{x}{n} \right)^n.
\end{equation}
By representing trigonometric functions in terms of exponentials, it is immediate to see that this function can be rewritten as:
\begin{equation}\label{form2}
F_n(x;\alpha) = \sum_{j=0}^n C_j(n;\alpha)e^{ix\lambda_{j,n}},
\end{equation}
where $C_j(n;\alpha) := \binom{n}{j} \left(\frac{\alpha+1}{2}\right)^{n-j} \left(\frac{1-\alpha}{2}\right)^j$ and $\lambda_{j,n} := 1-\frac{2j}{n}$.
Thus, $F_n$ is a simple superposition of waves with frequencies bounded between $-1$ and $1$. It is also evident that if $x$ is small, the function can be approximated by:
\begin{equation}\label{form3}
F_n(x;\alpha) \approx \left(1 + i\frac{\alpha x}{n}\right)^n \approx e^{i\alpha x}.
\end{equation}
Formulas \eqref{form2}, \eqref{form3} explain the term \textit{superoscillations}, as we have a function whose Fourier frequencies are bounded in $[-1,1]$, yet it approximates a function with an arbitrarily large frequency $\alpha$. To study this behavior, Berry introduced the notion of the \textit{local wave number} \cite{berry1994}, essentially the phase gradient, defined by:
\begin{equation}
k_{loc}(x) = \partial_x \text{arg}(F_n(x)) = \partial_x \text{Im}(\log(F_n(x))) = \frac{\alpha}{\cos^2 \frac{x}{n} + \alpha^2 \sin^2 \frac{x}{n}}.
\end{equation}
Superoscillations occur when $|k_{loc}(x)| > 1$, and it is $k_{loc}(x)$ that describes the strength of the effect. A simple calculation shows that the region of superoscillations is given by $|x| < n \arctan(1/\sqrt{\alpha})$ (assuming $\alpha > 1$), where $k_{loc}(x) > 1$.

Berry and his collaborators later explored whether superoscillations could represent fractal functions \cite{berry2016fractals, berry1980}. Specifically, in \cite{berry2016fractals}, Berry and Morley-Short consider the Weierstrass-Mandelbrot fractal function, showing that if it is truncated at an arbitrarily fine scale, it can be expressed to any desired accuracy using superoscillating functions.

The concrete example used in \cite{berry2016fractals} is the Weierstrass fractal defined as:
\begin{equation}
W(x, D, \gamma) = \sum_{n=0}^{\infty} \frac{\cos(\gamma^n x)}{\gamma^{n(2-D)}} \quad (\gamma > 1, 1 < D < 2),
\end{equation}
where $D$ is the fractal dimension and $\gamma$ is a scaling factor \cite{mandelbrot1982}. To create a band-limited approximation, they replace the high-frequency terms $\cos(\gamma^n x)$ with the real part:
\begin{equation}
S(x, \gamma^n, N) = \operatorname{Re}\left[ \left(\cos \frac{x}{N} + i \gamma^n \sin \frac{x}{N}\right)^N \right]
\end{equation}
of the canonical superoscillatory function \cite{ACSST11}. They demonstrate that this provides an excellent approximation of the truncated Weierstrass function. For the approximation to hold  over the interval $[-X, X]$, the parameter $N$ must satisfy $N \gg X^2(\gamma^{2n} - 1)$.

The authors propose a superoscillatory fractal function, $WS$, constructed by substituting $S$ into the Weierstrass sum:
\begin{equation}
WS(x, D, \gamma; X, K) = \sum_{n=0}^{\infty} \frac{S(x, \gamma^n, K(X, \gamma^n))}{\gamma^{n(2-D)}},
\end{equation}
where $K$ scales with the frequency $\gamma^n$ to ensure the approximation holds. Numerical simulations show that $WS$ faithfully reproduces the self-similar fractal structure of $W$ down to very fine scales. However, $WS$ grows exponentially outside the target interval, reaching values as high as $10^{11}$, implying extreme sensitivity to noise \cite{berry2017noise}.

Towards the end of their paper, the authors remark that ``...the precise mathematical sense in which $WS$ represents $W$ in the limit $K \to \infty$ would be an interesting project...''

This is the problem we address in this paper. We show that while $WS$ may not represent $W$ in the limit as $K \to \infty$ in a general sense, we can impose suitable conditions to ensure such a representation exists. We begin by analyzing the original Weierstrass function, constructed to prove the existence of a continuous but nowhere differentiable function \cite{weierstrass}:
\begin{equation}
f(x) = \sum_{m=0}^{\infty} a^m \cos(b^m \pi x),
\end{equation}
where $0 < a < 1$, $b$ is an odd integer, and $ab > 1 + \frac{3}{2}\pi$. In this work, we focus on the complex-valued version:
\begin{equation} \label{eq:weierstrass_complex}
W(x) = \sum_{m=0}^{\infty} a^m e^{i b^m \pi x}.
\end{equation}
We approximate the truncated series at index $N$:
\begin{equation}\label{troncata}
W_N(x) = \sum_{m=0}^{N} a^m e^{i b^m \pi x}.
\end{equation}
By replacing each high-frequency term $e^{i (b^m \pi) x}$ with the sequence $F_n(x; b^m \pi)$, we define:
\begin{equation} \label{eq:super_weierstrass}
\mathcal{W}_{N, n}(x) = \sum_{m=0}^{N} a^m \left[ \cos\left(\frac{x}{n}\right) + i (b^m \pi) \sin\left(\frac{x}{n}\right) \right]^n.
\end{equation}
As $n \to \infty$ with $N$ fixed, we recover the truncated Weierstrass function $W_N(x)$.

\medskip
We have structured this paper as follows. In Section \ref{sec:global_error} we establish some sharp and explicit error estimates for the superoscillating approximation of the truncated Weierstrass function: such estimates will be crucial in proving our main results. In Section \ref{sec:limits} we show that if we fix the superoscillating approximation to the exponential in the truncated Weierstrass function, and then we allow the truncation limit to go to infinity, we do not recover, in general the Weierstrass function; in other words, while the truncated Weierstrass function can be approximated by superoscillations, this is not true for the Weierstrass function itself. In Section \ref{SEC4}, however, we prove a theorem that we hope answers the original question of Berry and Morley-Short, and we demonstrate that if we allow the approximation index of the superoscillatory representation of the exponential, and the truncation index of the Weierstrass function, to both go to infinity in a precised proportion, then we do have a superoscillatory approximation of the Weierstrass function. The paper concludes with some remarks about the possibility of performing a similar procedure while employing different superoscillating sequences, and not necessarily the prototypical one described at the beginning of this introduction.

\section{Estimate of Truncated Weierstrass function with superoscillations} \label{sec:global_error}

Before proving the main approximation theorem, we establish a sharp, explicit estimate for the convergence of the superoscillating sequence to its limit.
We provide some preliminary elementary inequalities.

\begin{lemma}
Let $\gamma > 0$ and $y > -1$. Then the following estimate holds:
\[
|(1+y)^\gamma - 1| \leq \gamma |y| e^{\gamma |y|}
\]
\end{lemma}

\begin{proof}
We distinguish two cases based on the sign of $y$. Note that if $y=0$, both sides are 0 and the inequality holds trivially.

\paragraph{Case 1: $y > 0$.}
Since $y > 0$, we have $1+y \leq e^y$. Because $\gamma > 0$, we can raise both sides to the power of $\gamma$:
\[
(1+y)^\gamma \leq (e^y)^\gamma = e^{\gamma y}.
\]
Subtracting $1$ from both sides, we get:
\[
(1+y)^\gamma - 1 \leq e^{\gamma y} - 1.
\]
We invoke the inequality $e^x - 1 \leq x e^x$, which holds for all $x \geq 0$ (derived from $e^x - 1 = \int_0^x e^t dt \leq \int_0^x e^x dt = x e^x$). Setting $x = \gamma y$:
\[
(1+y)^\gamma - 1 \leq \gamma y e^{\gamma y}.
\]
Since $y > 0$, we have $y = |y|$ and $(1+y)^\gamma - 1 > 0$. Thus:
\[
|(1+y)^\gamma - 1| \leq \gamma |y| e^{\gamma |y|}.
\]

\paragraph{Case 2: $-1 < y < 0$.}
Let us use the Fundamental Theorem of Calculus. Define
$$f(t) = (1+t)^\gamma.
$$ Then:
\[
(1+y)^\gamma - 1 = f(y) - f(0) = \int_0^y f'(t) \, dt = \int_0^y \gamma(1+t)^{\gamma-1} \, dt.
\]
Taking the absolute value:
\[
|(1+y)^\gamma - 1| = \left| \int_0^y \gamma(1+t)^{\gamma-1} \, dt \right|.
\]
Since $y < 0$, the interval is $[y, 0]$. We can rewrite the absolute value of the integral as:
\[
\left| \int_0^y \gamma(1+t)^{\gamma-1} \, dt \right| = \gamma \int_y^0 (1+t)^{\gamma-1} \, dt.
\]
We assume $\gamma \geq 1$ for the uniform bound (as the singularity at $y=-1$ for $\gamma < 1$ requires separated treatment). If $\gamma \geq 1$, then $\gamma - 1 \geq 0$. For $t \in [y, 0]$, we have $1+t \leq 1$. Consequently:
\[
(1+t)^{\gamma-1} \leq 1^{\gamma-1} = 1.
\]
Substituting this bound into the integral:
\[
\gamma \int_y^0 (1+t)^{\gamma-1} \, dt \leq \gamma \int_y^0 1 \, dt = \gamma (0 - y) = \gamma (-y) = \gamma |y|.
\]
Since $e^{\gamma|y|} \geq 1$ for any real argument, we surely have $\gamma |y| \leq \gamma |y| e^{\gamma |y|}$. Thus:
\[
|(1+y)^\gamma - 1| \leq \gamma |y| e^{\gamma |y|}.
\]
Combining both cases, the estimate holds.
\end{proof}

\begin{theorem}[Explicit error estimate] \label{thm:carnot}
Let $M > 0$ be a fixed real number and $\alpha \in \mathbb{R}$. Define the constants:
\[
    K(\alpha) = \frac{|\alpha^2-1|M^2}{2} \quad \text{and} \quad J(\alpha) = 2|\alpha(1-\alpha^2)|M^3.
\]
Then, for any integer $n \ge \frac{4M}{\pi}$, the following explicit estimate holds for all $|x| \le M$:
\begin{equation}
    \left| \left( \cos\frac{x}{n} + i \alpha \sin\frac{x}{n} \right)^n - e^{i \alpha x} \right| \leq \frac{1}{n} \sqrt{ K(\alpha)^2 e^{2K(\alpha)/n} + \frac{J(\alpha)^2}{n^2} e^{K(\alpha)/n} }.
\end{equation}
\end{theorem}

\begin{proof}
Let $w := F_n(x, \alpha) = \left(\cos\frac{x}{n} + i\alpha \sin\frac{x}{n}\right)^n$ and $z := e^{i\alpha x}$. We aim to bound the Euclidean distance $|w-z|$ in the complex plane. Representing both numbers in polar form, $w = \rho_w e^{i\theta_w}$ and $z = \rho_z e^{i\theta_z}$, we have:
\begin{align*}
    \rho_w &= \left(\cos^2 \frac{x}{n} + \alpha^2 \sin^2 \frac{x}{n}\right)^{n/2} = \left( 1 + (\alpha^2-1)\sin^2 \frac{x}{n} \right)^{n/2}, \\
    \theta_w &= n \arctan \left(\alpha \tan \frac{x}{n}\right), \\
    \rho_z &= 1, \quad \theta_z = \alpha x.
\end{align*}
By the Carnot's Theorem in the complex plane:
\begin{equation} \label{eq:carnot_id}
    |w-z|^2 = \rho_w^2 + 1 - 2\rho_w \cos(\theta_w - \theta_z) = (\rho_w - 1)^2 + 2\rho_w(1 - \cos(\theta_w - \theta_z)).
\end{equation}
Using the standard Taylor inequality $1 - \cos y \leq \frac{y^2}{2}$, valid for all $y \in \mathbb{R}$, we obtain:
\begin{equation} \label{eq:main_ineq_proof}
    |w-z|^2 \leq (\rho_w - 1)^2 + \rho_w (\theta_w - \theta_z)^2.
\end{equation}

\medskip
We estimate of the modulus $\rho_w$
using the inequality $|(1+y)^\gamma - 1| \leq \gamma |y| e^{\gamma |y|}$ for $\gamma > 0$ and $y > -1$. Setting $\gamma = n/2$ and $y = (\alpha^2-1)\sin^2(x/n)$, and noting that $|\sin u| \le |u|$, we find:
\[
    |\rho_w - 1| \leq \frac{n}{2} |\alpha^2-1| \frac{x^2}{n^2} \exp\left( \frac{n}{2} |\alpha^2-1| \frac{x^2}{n^2} \right) \leq \frac{|\alpha^2-1|M^2}{2n} \exp\left( \frac{|\alpha^2-1|M^2}{2n} \right).
\]
By substituting $K(\alpha) = \frac{|\alpha^2-1|M^2}{2}$, we obtain:
\begin{equation} \label{eq:rho_est}
    |\rho_w - 1| \leq \frac{K(\alpha)}{n} e^{K(\alpha)/n}.
\end{equation}
Furthermore, since $\rho_w = (\rho_w - 1) + 1 \le |\rho_w - 1| + 1$, it follows from the growth of the exponential that $\rho_w \leq e^{K(\alpha)/n}$.

\medskip
Now we estimate of the phase difference $\Delta \theta = \theta_w - \theta_z$
define
$$h(u) = n \arctan(\alpha \tan u) - n \alpha u,
$$
 such that $\Delta \theta = h(x/n)$. Let $g(u) = \arctan(\alpha \tan u) - \alpha u$. Its derivative is:
\[
    g'(u) = \frac{1}{1 + (\alpha \tan u)^2} \cdot \alpha \sec^2 u - \alpha = \alpha \left( \frac{\sec^2 u - (1 + \alpha^2 \tan^2 u)}{1 + \alpha^2 \tan^2 u} \right) = \frac{\alpha(1-\alpha^2)\tan^2 u}{1+\alpha^2 \tan^2 u}.
\]
By the Mean Value Theorem, there exists $\xi$ between $0$ and $x/n$ such that
$$|g(x/n) - g(0)| = |x/n| |g'(\xi)|.
$$
 Since $n \ge 4M/\pi$, we have $|\xi| < |x/n| \le \pi/4$. In this range, $\tan^2 \xi \le 2\xi^2$ (since $\tan \xi \approx \xi$ and $1+\alpha^2 \tan^2 \xi \ge 1$). Thus:
\[
    |g'(\xi)| \leq |\alpha(1-\alpha^2)| \tan^2 \xi \leq 2 |\alpha(1-\alpha^2)| \xi^2 \leq 2 |\alpha(1-\alpha^2)| \frac{x^2}{n^2}.
\]
Then
$$|\Delta \theta| = n |g(x/n)| \leq n \cdot \frac{|x|}{n} \cdot 2 |\alpha(1-\alpha^2)| \frac{x^2}{n^2} \leq \frac{2 |\alpha(1-\alpha^2)| M^3}{n^2}.
$$
Setting $J(\alpha) = 2|\alpha(1-\alpha^2)|M^3$, we get $|\Delta \theta| \leq \frac{J(\alpha)}{n^2}$.

\medskip
Substituting the estimates into \eqref{eq:main_ineq_proof}:
\[
    |w-z|^2 \leq \left( \frac{K(\alpha)}{n} e^{K(\alpha)/n} \right)^2 + e^{K(\alpha)/n} \left( \frac{J(\alpha)}{n^2} \right)^2.
\]
Factoring out $1/n^2$ from the entire expression:
\[
    |w-z|^2 \leq \frac{1}{n^2} \left[ K(\alpha)^2 e^{2K(\alpha)/n} + \frac{J(\alpha)^2}{n^2} e^{K(\alpha)/n} \right].
\]
Taking the square root of both sides yields the desired inequality.
\end{proof}

We consider the complex-valued truncated Weierstrass function and its approximation via superoscillating sequences.

\begin{definition}[Truncated Weierstrass Function]
Let $0 < a < 1$ and $b$ be an odd integer such that $ab > 1 + \frac{3}{2}\pi$. For a fixed integer $N \in \mathbb{N}$, the truncated Weierstrass function $W_N(x)$ is defined as:
\begin{equation} \label{eq:truncated_weierstrass}
    W_N(x) = \sum_{m=0}^{N} a^m e^{i \alpha_m x}, \quad \text{where } \alpha_m = b^m \pi.
\end{equation}
\end{definition}

\begin{definition}[Superoscillating Approximation]
The superoscillating approximation of order $n$ for the truncated Weierstrass function is defined as:
\begin{equation} \label{eq:super_weierstrass}
    \mathcal{W}_{N, n}(x) = \sum_{m=0}^{N} a^m \left[ \cos\left(\frac{x}{n}\right) + i \alpha_m \sin\left(\frac{x}{n}\right) \right]^n.
\end{equation}
\end{definition}

It is easily verified that by the fundamental limit $\lim_{n \to \infty} (1 + \frac{i \alpha x}{n} + O(n^{-2}))^n = e^{i \alpha x}$, we recover the  function:
\begin{equation} \label{eq:super_weierstrassLIM}
    \lim_{n \to \infty} \mathcal{W}_{N, n}(x) = W_N(x).
\end{equation}

\begin{theorem}[Global Approximation Error]
Let $M > 0$ and $N$ be fixed. Let $\alpha_m = b^m \pi$ and define the maximal constant
$$K_{\max} = \frac{(\alpha_N^2 - 1)M^2}{2}.
$$
For any integer $n \ge \max\left( \frac{4M}{\pi}, K_{\max} \right)$,
the approximation error satisfies:
\begin{equation}
    \sup_{|x|\leq M} |W_N(x) - \mathcal{W}_{N,n}(x)| \leq \frac{e}{n} \mathcal{S}_1 + \frac{\sqrt{e}}{n^2} \mathcal{S}_2,
\end{equation}
where $\mathcal{S}_1$ and $\mathcal{S}_2$ are given by:
\begin{equation}\label{esse1}
    \mathcal{S}_1(N) = \frac{M^2}{2} \left[ \pi^2 \frac{(ab^2)^{N+1}-1}{ab^2-1} - \frac{1-a^{N+1}}{1-a} \right],
\end{equation}
\begin{equation}\label{esse2}
    \mathcal{S}_2(N) = 2M^3 \left[ \pi^3 \frac{(ab^3)^{N+1}-1}{ab^3-1} - \pi \frac{(ab)^{N+1}-1}{ab-1} \right].
\end{equation}
\end{theorem}

\begin{proof}
Define the point-wise error function $E_{N,n}(x) = |W_N(x) - \mathcal{W}_{N,n}(x)|$. By substituting definitions \eqref{eq:truncated_weierstrass} and \eqref{eq:super_weierstrass} and applying the triangle inequality, we bound the total error by the sum of individual term errors:
\begin{equation} \label{eq:triangle_step}
    E_{N,n}(x) \leq \sum_{m=0}^{N} a^m \left| e^{i \alpha_m x} - \left( \cos\frac{x}{n} + i \alpha_m \sin\frac{x}{n} \right)^n \right|.
\end{equation}
Applying Theorem \ref{thm:carnot} to each frequency $\alpha_m$, we have:
\[
    E_{N,n}(x) \leq \sum_{m=0}^{N} a^m \left[ \frac{1}{n} \sqrt{ K_m^2 e^{2K_m/n} + \frac{J_m^2}{n^2} e^{K_m/n} } \right],
\]
where $K_m = \frac{|\alpha_m^2-1|M^2}{2}$ and $J_m = 2|\alpha_m(1-\alpha_m^2)|M^3$.
Using the subadditivity of the square root, $\sqrt{A^2 + B^2} \leq A + B$, for nonnegative $A,B$, we obtain:
\[
    E_{N,n}(x) \leq \frac{1}{n} \sum_{m=0}^{N} a^m \left( K_m e^{K_m/n} + \frac{J_m}{n} e^{K_m/2n} \right).
\]
Under the constraint $n \ge K_{\max} \ge K_m$, it follows that $K_m/n \le 1$. Thus, $e^{K_m/n} \le e$ and $e^{K_m/2n} \le \sqrt{e}$. Factoring these constants out of the sum:
\[
    \sup_{|x|\le M} E_{N,n}(x) \leq \frac{e}{n} \sum_{m=0}^{N} a^m K_m + \frac{\sqrt{e}}{n^2} \sum_{m=0}^{N} a^m J_m.
\]

We now compute of $\mathcal{S}_1$ substituting $K_m = \frac{M^2}{2}(\pi^2 b^{2m} - 1)$:
\begin{align*}
    \mathcal{S}_1 &= \sum_{m=0}^{N} a^m \frac{M^2}{2}(\pi^2 b^{2m} - 1) = \frac{M^2 \pi^2}{2} \sum_{m=0}^{N} (ab^2)^m - \frac{M^2}{2} \sum_{m=0}^{N} a^m.
\end{align*}
Applying the geometric series formula $\sum_{k=0}^N r^k = \frac{r^{N+1}-1}{r-1}$:
\[
    \mathcal{S}_1 = \frac{M^2 \pi^2}{2} \left( \frac{(ab^2)^{N+1}-1}{ab^2-1} \right) - \frac{M^2}{2} \left( \frac{a^{N+1}-1}{a-1} \right).
\]
Factoring $M^2/2$ and adjusting signs leads to the expression in the theorem.
To compute $\mathcal{S}_2$ we
substitute  $J_m = 2M^3(\pi^3 b^{3m} - \pi b^m)$:
\begin{align*}
    \mathcal{S}_2 &= \sum_{m=0}^{N} a^m [ 2M^3 (\pi^3 b^{3m} - \pi b^m) ] = 2M^3 \pi^3 \sum_{m=0}^{N} (ab^3)^m - 2M^3 \pi \sum_{m=0}^{N} (ab)^m.
\end{align*}
Using the geometric series summation once more:
\[
    \mathcal{S}_2 = 2M^3 \pi^3 \left( \frac{(ab^3)^{N+1}-1}{ab^3-1} \right) - 2M^3 \pi \left( \frac{(ab)^{N+1}-1}{ab-1} \right).
\]
This completes the detailed derivation.
\end{proof}

\section{Divergence of Fixed-$n$ Limit} \label{sec:limits}
In this section we show that while the truncated Weierstrass function can be approximated with superoscillations, the Weierstrass function cannot if one keeps the superoscillating index $n$ fixed.

\begin{theorem}[Divergence of the Limit $N \to \infty$] \label{thm:divergence}
Let $n \geq 1$ be a fixed integer and let $M > 0$ such that $M < n\pi$. Consider the sequence of functions defined by the partial sums:
\begin{equation} \label{eq:super_weierstrass_sum}
    \mathcal{W}_{N, n}(x) = \sum_{m=0}^{N} a^m \left[ \cos\left(\frac{x}{n}\right) + i \alpha_m \sin\left(\frac{x}{n}\right) \right]^n,
\end{equation}
where $\alpha_m = b^m \pi$, $0 < a < 1$, and $b > 1$.
Then, the infinite series $\lim_{N \to \infty} \mathcal{W}_{N, n}(x)$ behaves as follows:
\begin{enumerate}
    \item For $x = 0$, the series converges absolutely to $\frac{1}{1-a}$.
    \item For any $x \in [-M, M] \setminus \{0\}$, the series diverges.
\end{enumerate}
\end{theorem}

\begin{proof}
Let $u_m(x)$ denote the general term of the series:
\[
    u_m(x) = a^m \left[ \cos\left(\frac{x}{n}\right) + i b^m \pi \sin\left(\frac{x}{n}\right) \right]^n.
\]

\paragraph{Case 1: $x = 0$.}
Substituting $x=0$ into the expression for $u_m(x)$, we note that $\cos(0)=1$ and $\sin(0)=0$. The term simplifies to:
\[ u_m(0) = a^m [ 1 + i(0) ]^n = a^m \cdot 1^n = a^m. \]
The series becomes $\sum_{m=0}^{\infty} a^m$. Since $0 < a < 1$, this is a convergent geometric series. Its sum is given by the well-known formula:
\[ \mathcal{W}_{\infty, n}(0) = \sum_{m=0}^{\infty} a^m = \frac{1}{1-a}. \]

\paragraph{Case 2: $x \neq 0$.}
Assume $x \in [-M, M] \setminus \{0\}$. Since $M < n\pi$, we have $0 < |x/n| < \pi$, which implies that $\sin(x/n) \neq 0$. To study the convergence of $\sum u_m(x)$, we examine the limit of the modulus of the general term using the Ratio Test. Let $\rho_m(x)$ be the ratio of the moduli of consecutive terms:
\[
    \rho_m(x) = \left| \frac{u_{m+1}(x)}{u_m(x)} \right| = \frac{a^{m+1}}{a^m} \left| \frac{\cos(x/n) + i b^{m+1} \pi \sin(x/n)}{\cos(x/n) + i b^m \pi \sin(x/n)} \right|^n.
\]
We simplify the expression inside the modulus:
\[
    \rho_m(x) = a \left( \frac{|\cos(x/n) + i b^{m+1} \pi \sin(x/n)|}{|\cos(x/n) + i b^m \pi \sin(x/n)|} \right)^n,
\]
and
\[
    \rho_m(x) = a \left( \sqrt{ \frac{\cos^2(x/n) + b^{2m+2} \pi^2 \sin^2(x/n)}{\cos^2(x/n) + b^{2m} \pi^2 \sin^2(x/n)} } \right)^n.
\]
To find the limit as $m \to \infty$, we divide both the numerator and the denominator inside the square root by the dominant term $b^{2m}$:
\[
    \rho_m(x) = a \left( \sqrt{ \frac{\frac{\cos^2(x/n)}{b^{2m}} + b^2 \pi^2 \sin^2(x/n)}{\frac{\cos^2(x/n)}{b^{2m}} + \pi^2 \sin^2(x/n)} } \right)^n.
\]
As $m \to \infty$, the terms involving $1/b^{2m}$ vanish since $b > 1$. Thus:
\[
    L = \lim_{m \to \infty} \rho_m(x) = a \left( \sqrt{ \frac{b^2 \pi^2 \sin^2(x/n)}{\pi^2 \sin^2(x/n)} } \right)^n = a ( \sqrt{b^2} )^n = a b^n.
\]
Given the Weierstrass condition $ab > 1 + \frac{3}{2}\pi$, it is clear that $ab > 1$. Since $b > 1$ and $n \ge 1$, we have:
\[ L = a b^n \ge ab > 1. \]
According to the Ratio Test, if the limit of the ratio of the moduli is strictly greater than 1, the series $\mathcal{W}_{N, n}(x)$ diverges for any $x \neq 0$.
\end{proof}

We investigate the sequential limit where the oscillation parameter $n$ is taken to infinity before the truncation order $N$. This order of operations corresponds to approximating each individual frequency component before summing the entire series.

\begin{theorem}[Convergence of the Iterated Limit]
Let $\mathcal{W}_{N, n}(x)$ be the superoscillating approximation of the Weierstrass function defined by:
\begin{equation} \label{eq:super_weierstrass_def}
    \mathcal{W}_{N, n}(x) = \sum_{m=0}^{N} a^m \left[ \cos\left(\frac{x}{n}\right) + i (b^m \pi) \sin\left(\frac{x}{n}\right) \right]^n,
\end{equation}
where $0 < a < 1$ and $b > 1$. Then, for any $x \in \mathbb{R}$, the iterated limit converges to the Weierstrass function $W(x) = \sum_{m=0}^{\infty} a^m e^{i b^m \pi x}$:
\begin{equation}
    \lim_{N \to \infty} \left( \lim_{n \to \infty} \mathcal{W}_{N, n}(x) \right) = W(x).
\end{equation}
\end{theorem}

\begin{proof}
The proof is immediate.
\end{proof}

\begin{remark}[Non-Commutativity of Limits]
It is important to emphasize that the order of limits is not interchangeable. As established in previous results, the limit $\lim_{N \to \infty} \mathcal{W}_{N, n}(x)$ diverges for any fixed $n$ and $x \neq 0$. Therefore:
\[
    \lim_{N \to \infty} \lim_{n \to \infty} \mathcal{W}_{N, n}(x) \neq \lim_{n \to \infty} \lim_{N \to \infty} \mathcal{W}_{N, n}(x),
\]
illustrating a failure due to the lack of uniform convergence in $n$ as $N \to \infty$.
\end{remark}

\section{The Joint Convergence Theorem}\label{SEC4}

In this section, we answer the original question asked by M. V. Berry in \cite{berry2016fractals} and we establish the conditions under which the superoscillating approximation $\mathcal{W}_{N, n}(x)$ converges to the Weierstrass function $W(x)$ as both the truncation order $N$ and the oscillation parameter $n$ tend to infinity simultaneously.

\begin{theorem}[Joint Convergence]
Let $M > 0$, $0 < a < 1$, and $b > 1$ be an odd integer. Let $W(x) = \sum_{m=0}^{\infty} a^m e^{i b^m \pi x}$ be the Weierstrass function. If the sequence of integers $\{n_N\}_{N \in \mathbb{N}}$ satisfies the growth condition
\begin{equation} \label{eq:growth_condition}
    \lim_{N \to \infty} \frac{(ab^3)^N}{n_N} = 0,
\end{equation}
then the sequence of functions $\mathcal{W}_{N, n_N}(x)$ converges uniformly to $W(x)$ on the interval $[-M, M]$:
\begin{equation}
    \lim_{N \to \infty} \sup_{|x| \le M} \left| \mathcal{W}_{N, n_N}(x) - W(x) \right| = 0.
\end{equation}
\end{theorem}

\begin{proof}
Let $x \in [-M, M]$. We define the total error as $\mathcal{E}(N) = \left| \mathcal{W}_{N, n_N}(x) - W(x) \right|$. By the triangle inequality, we decompose this error into two parts:
\begin{equation} \label{eq:triangle_split}
    \mathcal{E}(N) \leq \underbrace{\left| \mathcal{W}_{N, n_N}(x) - W_N(x) \right|}_{E_1(N)} + \underbrace{\left| W_N(x) - W(x) \right|}_{E_2(N)},
\end{equation}
where $W_N(x) = \sum_{m=0}^{N} a^m e^{i b^m \pi x}$ is the $N$-th partial sum of the Weierstrass series.

\medskip
 Estimate of the Truncation Error ($E_2$).
The term $E_2(N)$ represents the tail of the Weierstrass series. Using the fact that $|e^{i \theta}| = 1$ for any $\theta \in \mathbb{R}$:
\[
    E_2(N) = \left| \sum_{m=N+1}^{\infty} a^m e^{i b^m \pi x} \right| \leq \sum_{m=N+1}^{\infty} a^m = \frac{a^{N+1}}{1-a}.
\]
Since $0 < a < 1$, the right-hand side is a null sequence, implying $\lim_{N \to \infty} E_2(N) = 0$ uniformly on $[-M, M]$.

\medskip
 Estimate of the Approximation Error ($E_1$).
From the previously established Global Approximation Error theorem, for $n_N \ge \max(\frac{4M}{\pi}, K_{\max})$, we have:
\[
    E_1(N) \leq \frac{e}{n_N} \mathcal{S}_1(N) + \frac{\sqrt{e}}{n_N^2} \mathcal{S}_2(N).
\]
We examine the asymptotic growth of $\mathcal{S}_1(N)$ and $\mathcal{S}_2(N)$. Recalling their definitions, see also \eqref{esse1}, \eqref{esse2}:
\begin{align*}
    \mathcal{S}_1(N) &= \frac{M^2}{2} \left[ \pi^2 \frac{(ab^2)^{N+1}-1}{ab^2-1} - \frac{1-a^{N+1}}{1-a} \right] \sim O\left((ab^2)^N\right), \\
    \mathcal{S}_2(N) &= 2M^3 \left[ \pi^3 \frac{(ab^3)^{N+1}-1}{ab^3-1} - \pi \frac{(ab)^{N+1}-1}{ab-1} \right] \sim O\left((ab^3)^N\right).
\end{align*}
For $b > 1$, the growth of $(ab^3)^N$ strictly dominates $(ab^2)^N$. Therefore, there exist positive constants $C_1, C_2$ such that for sufficiently large $N$:
\[
    E_1(N) \leq \frac{C_1 (ab^2)^N}{n_N} + \frac{C_2 (ab^3)^N}{n_N^2}.
\]
Since $n_N$ must grow at least as fast as $(ab^3)^N$ to satisfy \eqref{eq:growth_condition}, the term $1/n_N^2$ decays much faster than $1/n_N$. We can thus bound $E_1(N)$ by a single dominant term:
\[
    E_1(N) \leq C \frac{(ab^3)^N}{n_N},
\]
for some constant $C > 0$. By the hypothesis \eqref{eq:growth_condition}, it follows that $\lim_{N \to \infty} E_1(N) = 0$.

\medskip
Combining the two estimates into \eqref{eq:triangle_split}, we obtain:
\[
    \sup_{|x| \le M} \mathcal{E}(N) \leq C \frac{(ab^3)^N}{n_N} + \frac{a^{N+1}}{1-a}.
\]
As $N \to \infty$, both terms on the right-hand side vanish. Thus, the superoscillating sequence $\mathcal{W}_{N, n_N}(x)$ converges uniformly to the Weierstrass function $W(x)$ on $[-M, M]$.
\end{proof}

\subsection{Interpretation of the Stability Limit}

Consider the function:
\begin{equation} \label{eq:weierstrass}
    W(x) = \sum_{m=0}^{\infty} a^m e^{i b^m \pi x}.
\end{equation}
To approximate $W(x)$ using band-limited sequences, we define the superoscillating approximant of order $(N, n)$ as:
\begin{equation} \label{eq:super_weierstrass_def}
     \mathcal{W}_{N, n}(x)= \sum_{m=0}^{N} a^m \left[ \cos\left(\frac{x}{n}\right) + i (b^m \pi) \sin\left(\frac{x}{n}\right) \right]^n.
\end{equation}
Applying the \textit{Joint Convergence Theorem} to the function $\mathcal{W}_{N, n}(x)$ requires a delicate balance between the spectral resolution $N$ (which also corresponds to the truncation of the series expressing the Weierstrass function) and the computational complexity $n$.

Uniform convergence as $N, n \to \infty$ is governed by the ratio between the oscillation index $n_N$ and the spectral growth $(ab^3)^N$. Formally, convergence is guaranteed if this ratio is infinitesimal in the limit. By choosing, for example, $n_N = N(ab^3)^N$ we obtain:
\begin{equation} \label{eq:experiment_condition}
    \lim_{N \to \infty} \frac{(ab^3)^N}{n_N} = \lim_{N \to \infty} \frac{1}{N} = 0.
\end{equation}

The concept of the \textit{Divergence Wall} marks the transition between numerical stability and exponential explosion. This critical threshold is defined by the asymptotic condition $n_N \sim (ab^3)^N$. The nature of the convergence is determined by the behavior of the ratio $R_N = \frac{(ab^3)^N}{n_N}$:

\medskip
 {\em Sub-Critical Regime ($R_N \to \infty$):} The complexity $n$ grows too slowly relative to the sampled frequencies. The approximant enters a region of exponential instability where the error diverges instantaneously.

\medskip
{\em Critical Regime ($R_N \to C > 0$):} The system sits on a knife-edge. The error remains bounded, but the sequence fails to converge uniformly to zero, thus failing to capture the fractal details.

\medskip
{\em Super-Critical Regime ($R_N \to 0$):} Complexity dominates spectral growth. The superoscillatory structure effectively suppresses the growth of high-frequency terms, ensuring uniform convergence on any compact interval $[-M, M]$.

\begin{figure}[htbp]
    \centering
    \includegraphics[width=0.7\textwidth]{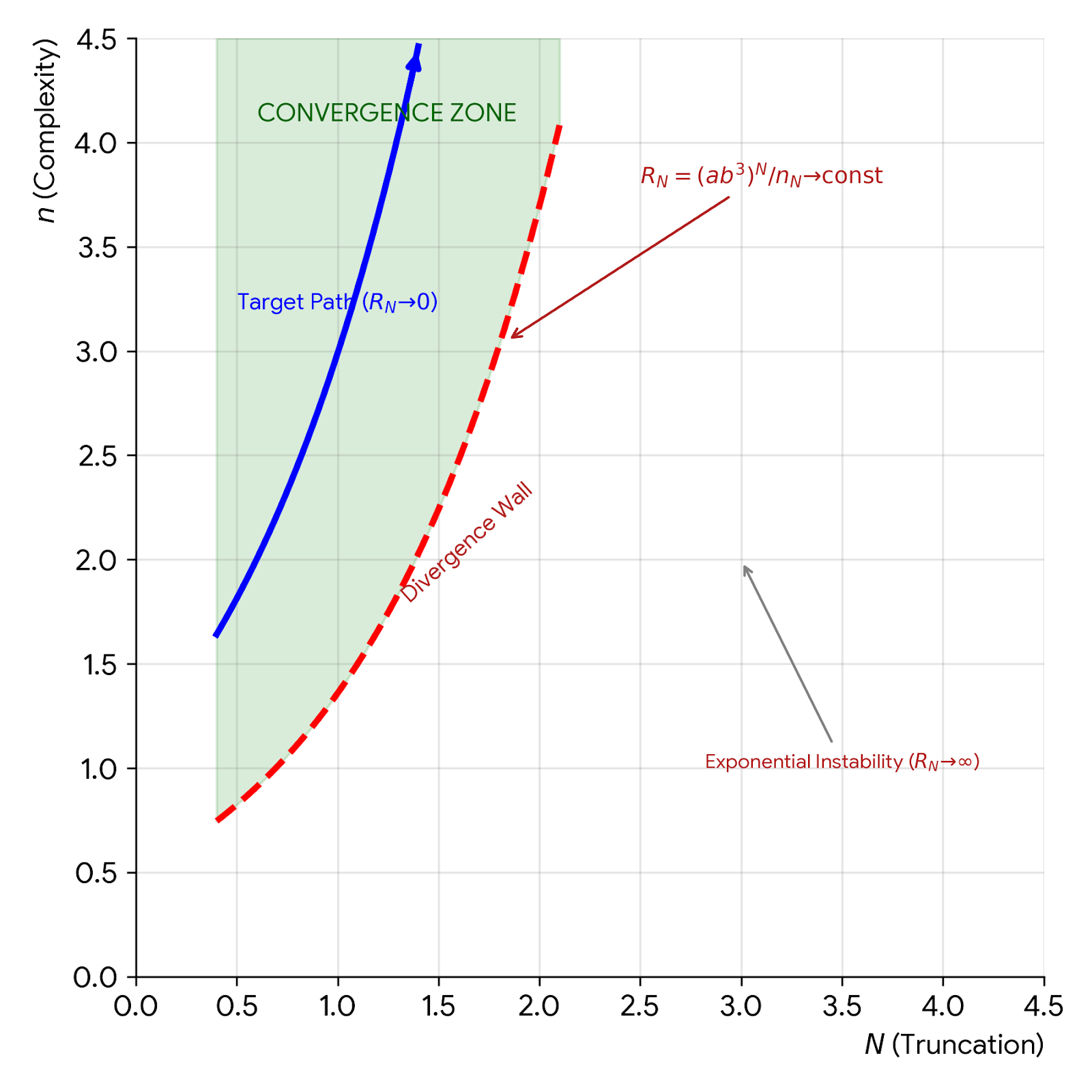}
    \caption{Stability phase diagram showing the Divergence Wall and the region of uniform convergence.}
    \label{fig:stability_diagram}
\end{figure}
\section{Concluding Remarks}\label{SECREM}

The results obtained in this paper provide a Bernstein-type approximation of the Weierstrass function.
A natural extension for future research is the study of approximations based on
Lagrange-type superoscillations. This approach is more complicated due to the nature of Lagrange interpolation.

The aim of our future work is to replace the high-frequency exponentials $e^{i \alpha_m x}$ in the Weierstrass function, with functions whose spectrum is strictly contained within the band-limited interval $[-1, 1]$ and is given by Lagrange-type superoscillations.

\medskip

Such superoscillations arise when we consider a set
 $\{h_j(n)\}_{j=0}^n$, $n \in \mathbb{N}$,  of distinct nodes in $[-1,1]$ and $\alpha \in \mathbb{R}$ such that $|\alpha| > 1$. In \cite{ACSSST21}, we introduced the sequence
 \begin{equation}
     T_n(x; \alpha)= \sum_{j=0}^n \left[ \prod_{k=0, k \neq j}^n \left( \frac{h_k(n) - \alpha}{h_k(n) - h_j(n)} \right) \right] e^{i h_j(n) x}.
\end{equation}
and in \cite{BOREL} we proved that $T_n(x; \alpha)$ is superoscillating and converges to $e^{i\alpha x}$.

  The next natural step consists in substituting each exponential term $e^{i b^m \pi x}$ in $W_N(x)$ with the superoscillating sequence $T_n(x; b^m \pi)$, and to investigate the superoscillating approximation of the truncated Weierstrass function, denoted by $\mathcal{W}_{N,n}(x)$, defined as:
\begin{equation}
    \mathcal{W}_{N,n}(x) = \sum_{m=0}^{N} a^m T_n(x; b^m \pi).
\end{equation}
Using the Lagrange type coefficients, the explicit form of the approximation is:
\begin{equation}
    \mathcal{W}_{N,n}(x) = \sum_{m=0}^{N} a^m \sum_{j=0}^n \left[ \prod_{k=0, k \neq j}^n \left( \frac{h_k(n) - b^m \pi}{h_k(n) - h_j(n)} \right) \right] e^{i h_j(n) x}.
\end{equation}
The key difficulty in this investigation will arise from the fact that we do not have a closed form expression for the Lagrange type approximation of superoscillations.

\medskip\medskip
{\bf Conflict of Interest}. The authors declare that they have no competing interests regarding the publication of this paper.

{\bf Author contributions}. All authors contributed equally to the study, read and approved the final version of the submitted manuscript.

{\bf Availability of data}. There are no data associated with the research in this paper.

\end{document}